\documentclass{amsart} 

\usepackage[utf8]{inputenc} 
\usepackage{bbm}
\usepackage{amscd,amssymb,amsmath,amsthm}
\usepackage{latexsym}

\usepackage{parskip}

\usepackage{graphics}
\usepackage{graphicx}
\usepackage{pgfplots}
\pgfplotsset{my style/.append style={axis x line=middle, axis y line=
middle, xlabel={$x$}, ylabel={$y$}, axis equal }}

\theoremstyle{plain}
\newtheorem{lemma}{Lemma}
\newtheorem{theorem}{Theorem}

\newtheorem{claim}{Claim}
\numberwithin{theorem}{section}
\numberwithin{proposition}{section}
\numberwithin{lemma}{section}
\numberwithin{corollary}{section}
\numberwithin{claim}{section}

\theoremstyle{definition}
\newtheorem{definition}{Definition}
\newtheorem{example}{Example}

\newtheorem{remark}{Remark}
\numberwithin{definition}{section}
\numberwithin{example}{section}
\numberwithin{question}{section}
\numberwithin{remark}{section}

\newcommand{\norm}[1]{\left\Vert #1\right\Vert}
\newcommand{\R}{\mathbb R}

\newcommand{\N}{\mathbb{N}}
\newcommand{\e}{\varepsilon}
\newcommand{\p}{\varphi}
\newcommand{\al}{\alpha}

\newcommand{\ii}{\infty}
\newcommand \beq{\begin{eqnarray*}}
\newcommand \eeq{\end{eqnarray*}}

\newcommand{\weak}{\rightharpoonup}

\usepackage{pdfpages} 

\title{A weak convergence theorem for mean nonexpansive mappings}
\author{Torrey M. Gallagher}
\address{University of Pittsburgh, United States}
\email{tmg34@pitt.edu}
\thanks{The present author would like to thank Chris Lennard for his helpful suggestions regarding the preparation of this paper.}

\date{April 4, 2016}

\keywords{Mean nonexpansive, Opial's property, asymptotically regular, weak convergence}

\begin{document}
\maketitle

\begin{abstract}
In this paper, we prove first that the iterates of a mean nonexpansive map defined on a weakly compact, convex set converge weakly to a fixed point in the presence of Opial's property and asymptotic regularity at a point.  Next, we prove the analogous result for closed, convex (not necessarily bounded) subsets of uniformly convex Opial spaces.  These results generalize the classical theorems for nonexpansive maps of Browder and Petryshyn in Hilbert space and Opial in reflexive spaces satisfying Opial's condition.
\end{abstract}

\section{Introduction}
Let $(X,\norm{\cdot})$ be a Banach space. Given $C \subseteq X$, we say a function $T: C \to X$ is \textit{nonexpansive} if 
\[
\norm{Tx-Ty}\leq \norm{x-y}
\]
for all $x,y \in C$.   It is a well-known application of Banach's contraction mapping theorem that every nonexpansive map $T : C\to C$ has an \textit{approximate fixed point sequence}; that is, a sequence $(u_n)_n$ in $C$ for which $\norm{Tu_n - u_n} \to_n 0$. Also, we say $T:C \to C$ is \textit{asymptotically regular at $x$} if
\[
\lim_{n\to\ii} \norm{T^n x - T^{n+1}x} = 0.
\]
If $T$ is asymptotically regular at every $x \in C$, we simply say $T$ is asymptotically regular.  Note that asymptotic regularity at $x$ implies that $(T^n x)_n$ is an approximate fixed point sequence for $T$.  Denote the set of all fixed points of $T$ as $F(T)$.  That is, $F(T) := \{x \in C : Tx=x\}$.

In 1966, Browder and Petryshyn \cite[Theorem 4]{browderpetryshyn66} proved the following theorem for asymptotically regular nonexpansive mappings on Hilbert space.

\begin{theorem}[Browder and Petryshyn]
Suppose $H$ is a Hilbert space and $T: H \to H$ is nonexpansive and asymptotically regular with $F(T)=\{x_0\}$.  Then $(T^nx)_n$ converges weakly to $x_0$.
\end{theorem}

In 1967, Opial \cite{opial67} extended this theorem to spaces satisfying Opial's property.  Recall that $C \subseteq X$ has the \textit{Opial property} if, whenever $(u_n)_n$ is a sequence in $C$ converging weakly to some $u \in X$, we have
\[
\liminf_n \norm{u_n - u} < \liminf_n \norm{u_n - v}
\]
for any $v \neq u$.  All Hilbert spaces have the Opial property, as do $\ell^p$ spaces for all $p \in (1, \ii)$.  $L^p$ fails to have the Opial property for all $p\neq 2$, however.

We will further extend Opial's result to the class of mean nonexpansive maps, first defined by Goebel and Jap\'on Pineda in 2007 \cite{gjp07}.  We say $T : C \to C$ is \textit{mean nonexpansive} (or $\al$-nonexpansive) if, for some multi-index $\al=(\al_1,\ldots,\al_n)$ with $\al_1, \al_n >0$, $\al_j \geq 0$ for all $j$, and $\al_1+\cdots +\al_n = 1$, we have
\[
\sum_{j=1}^n \al_j \norm{T^jx - T^jy} \leq \norm{x-y}
\]
for all $x, y \in C$.  Goebel and Jap\'on Pineda further suggested the notion of \textit{$(\al,p)$}-nonexpansiveness, wherein $T$ would satisfy
\[
\sum_{j=1}^n \al_j \norm{T^jx - T^jy}^p \leq \norm{x-y}^p
\]
for some $p \in [1,\ii)$.  It is easy to check that any $(\al,p)$-nonexpansive map is mean nonexpansive (i.e. $(\al,1)$-nonexpansive), but the converse does not necessarily hold \cite{piasecki13}.

To prove our theorem, we will need one further notion.  We will use ``$\weak$'' to denote weak convergence and ``$\to$'' to denote strong convergence.  We say $T: C \to X$ is \textit{demiclosed at $y$} if, whenever  $x_n \weak x$ in $C$ and $Tx_n \to y$, it follows that $Tx=y$.  The present author recently proved \cite[Theorem 4.2]{gallagher16} that if $C \subseteq X$ is closed and convex and has the Opial property, then any mean nonexpansive map $T: C\to C$ is demiclosed at 0.  We will use this demiclosedness principle to extend the theorems of Browder and Petryshyn and Opial stated above.  First, we will present the results for the simple case of multi-indices of length 2 before proving the full theorem for multi-indices of arbitrary length.


\section{Results for $\al=(\al_1,\al_2)$}

Let us state the main theorem of this section.  The proofs of the following theorem and lemmas can be found in the next section.

\begin{theorem}\label{theorem}
Suppose $(X, \norm{\cdot})$ is a Banach space and $C \subseteq X$ is weakly compact, convex, and has the Opial property.  Suppose further that $T : C \to C$ is $(\al_1,\al_2)$-nonexpansive and asymptotically regular at some point $x \in C$. Then $(T^n x)_n$ converges weakly to a fixed point of $T$.
\end{theorem}

To ensure that this theorem is a genuine extension of the classical theorems for nonexpansive maps, we present an example of a $((\frac{1}{2},\frac{1}{2}), 2)$-nonexpansive (hence mean nonexpansive) map defined on $(\ell^2,\norm{\cdot}_2)$ for which none of its iterates are nonexpansive. The map below is based on an example given by Goebel and Sims \cite{goebelsims10} and can also be found in \cite{gallagher16}; moreover, it is asymptotically regular.

\begin{example}
Let $(\ell^2,\norm{\cdot}_2)$ be the Hilbert space of square-summable sequences endowed with its usual norm.  Let $\tau : [-1,1]\to[-1,1]$ be given by
\[
\tau(t) :=
\begin{cases}
\sqrt{2} \, t + (\sqrt{2}-1) & -1 \leq t \leq -\frac{\sqrt{2}-1}{\sqrt{2}}\\
0 & -\frac{1+\sqrt{2}}{\sqrt{2}} \leq t \leq \frac{1+\sqrt{2}}{\sqrt{2}}\\
\sqrt{2} \, t - (\sqrt{2}-1) & \frac{\sqrt{2}-1}{\sqrt{2}} \leq t \leq 1
\end{cases}
\]
and note the following facts about $\tau$: 
\begin{enumerate}
\item $\tau$ is Lipschitz with $k(\tau) = \sqrt{2}$,
\item $|\tau(t)| \leq |t|$ for all $t \in [-1,1]$, and
\end{enumerate} 
Let $B_{\ell^2}$ denote the closed unit ball of $(\ell^2,\norm{\cdot}_2)$ and for any $x \in \ell^2$, define $T$ by
\[
T(x_1,x_2,\ldots) := \left(\tau(x_2),\sqrt{\frac{2}{3}}\,x_3, x_4, x_5, \ldots\right)
\]
and
\[
T^2 (x_1,x_2,\ldots) = \left( \tau\left( \sqrt{\frac{2}{3}} \, x_3 \right),  \sqrt{\frac{2}{3}} \, x_4, x_5, \ldots \right)
\]
Observe that $|\tau(t)|\leq|t|$ implies that $T(B_{\ell^2}) \subseteq B_{\ell^2}$, and $k(T)=\sqrt{2}>1$ and $k(T^j)= \frac{2}{\sqrt{3}}>1$ for all $j\geq 2$.  Now, for any $x, y \in B_{\ell^2}$ we find
\begin{align*}
\frac{1}{2} \norm{Tx-Ty}_2^2 &+ \frac{1}{2}\norm{T^2x-T^2y}_2^2\\
	&= \frac{1}{2}\left( |\tau(x_2)-\tau(y_2)|^2 + \frac{2}{3} |x_3 - y_3|^2 + \sum_{j=4}^\ii |x_j - y_j|^2 \right)\\
	&+ \frac{1}{2}\left( \Bigg|\tau\left( \sqrt{\frac{2}{3}} \, x_3\right)-\tau\left(\sqrt{\frac{2}{3}} \, y_3\right)\Bigg|^2 + \frac{2}{3} |x_4 - y_4|^2 + \sum_{j=5}^\ii |x_j - y_j|^2 \right)\\
	&\leq \frac{1}{2} \left( 2 |x_2 - y_2|^2 + \frac{4}{3}|x_3 - y_3|^2 + \frac{5}{3}|x_4-y_4|^2 + 2 \sum_{j=5}^\ii |x_j-y_j|^2\right)\\
	&\leq \norm{x-y}_2^2
\end{align*}
Hence, $T: B_{\ell^2} \to B_{\ell^2}$ is a $((\frac{1}{2},\frac{1}{2}), 2)$-nonexpansive map for which each iterate $T^j$ is not nonexpansive.
\end{example}

Before proving the theorem, let's state some preliminary definitions and results.  For any $x \in C$, let 
\[
\omega_w (x) := \{ y \in C : y \mbox{ is a weak subsequential limit of } (T^nx)_n \}
\]
and note that if $C$ is weakly compact, then $\omega_w (x) \neq \emptyset$.  Further note that if $I-T$ is demiclosed at $0$ and asymptotically regular at $x$, then $\emptyset \neq \omega_w (x) \subseteq F(T)$. We have the following lemma.

\begin{lemma}\label{existence}
Suppose $C$ is weakly compact and convex with the Opial property, and suppose that $T : C \to C$ is $(\al_1,\al_2)$-nonexpansive and asymptotically regular at some $x \in C$.  Then for all $y \in \omega_w (x)$, $\lim_n \norm{T^nx - y}$ exists.
\end{lemma}

Our theorem will be proved if we can show that $\omega_w (x)$ is a singleton.  This follows from the fact that $C$ is Opial and the knowledge that $(\norm{T^nx - y})_n$ converges for all $y \in \omega_w (x)$, as summarized in the following lemma.\\

\begin{lemma}\label{uniqueness}
If $C \subseteq X$ is Opial, $T : C\to C$ is a function, and for some $x \in C$, $\lim_n \norm{T^nx - y}$ exists for all $y \in \omega_w (x)$, then $\omega_w (x)$ is empty or consists of a single point.
\end{lemma}


\section{Proofs}
\begin{proof}[Proof of Lemma \ref{existence}]
$C$ closed and convex with the Opial property implies that $I-T$ is demiclosed at 0.  That is, whenever $(z_n)_n$ is a sequence in $C$ converging weakly to some $z$ (which is necessarily in $C$ since closed and convex implies weakly closed) for which $\norm{(I-T)z_n} \to_n 0$, it follows that $(I-T)z = 0$.

By the asymptotic regularity of $T$ at $x$, we have that $(T^nx)_n$ is an approximate fixed point sequence for $T$.

Since $y \in \omega_w (x)$ and $I-T$ is demiclosed at $0$, we have that $y$ is a fixed point of $T$ and we see that
\begin{align*}
\al_1 \norm{Tx - y} + \al_2 \norm{T^{2}x -y } &= \al_1 \norm{Tx - Ty} + \al_2 \norm{T^{2}x -T^{2}y }\\
&\leq \norm{x - y}
\end{align*}

Hence, at least one of $\norm{Tx-y}$ or $\norm{T^2x-y}$ must be less than or equal to $\norm{x-y}$.  Let $k_1 \in \{1,2\}$ be such that $\norm{T^{k_1}x-y} \leq \norm{x-y}$.

Next, we know that
\begin{align*}
\al_1 \norm{T^{k_1+1}x - y} + \al_2 \norm{T^{k_1+2}x -y } &= \al_1 \norm{T^{k_1+1}x - T^{k_1+1}y} + \al_2 \norm{T^{k_1+2}x -T^{k_1+2}y }\\
&\leq \norm{T^{k_1}x - T^{k_1} y}\\
&= \norm{T^{k_1}x - y}
\end{align*}
and so one of $\norm{T^{k_1+1}x - y}$ or $\norm{T^{k_1+2}x - y}$ must be less than or equal to $\norm{T^{k_1}x - y}$.  As above, let $k_2 \in \{k_1+1, k_1+2\}$ be such that $\norm{T^{k_2}x-y} \leq \norm{T^{k_1}x-y}$.

Inductively, build a sequence $(k_n)_n$ which satisfies 
\begin{enumerate}
\item $k_n + 1 \leq k_{n+1} \leq k_n + 2$, and
\item $\norm{T^{k_{n+1}}x-y} \leq \norm{T^{k_n}x-y}$
\end{enumerate}
for all $n \in \N$.  Now $(\norm{T^{k_n}x-y})_n$ is a non-increasing sequence in $\R^+$, and is thus convergent to some $q\in \R^+$.  

Consider the set $M := \N \setminus \{k_n : n \in \N \}$.  We have two cases.  First, if $M$ is a finite set, then the claim is proved.  Second, if $M$ is infinite, write $M = \{ m_n : n \in \N \}$, where $(m_n)_n$ is strictly increasing.  Note that, by property (1) of the sequence $(k_n)_n$ above, we must have that for all $n \in \N$, there exists a $j_n \in \N$ for which
\[
m_n = k_{j_n}+1
\]

Also, $(j_n)_n$ is strictly increasing.  Asymptotic regularity of $T$ at $x$ and the fact that $\lim_n \norm{T^{k_n}x-y} = q$ gives us that for any $\e >0$, there is $n$ large enough such that 
\begin{enumerate}
\item $\norm{T^{m_n}x - T^{m_n-1}x} < \e/2$, and
\item $\big|\, \norm{T^{k_{j_n}}x - y} - q \,\big| < \e/2$.
\end{enumerate}

Thus, 
\begin{align*}
\norm{T^{m_n}x - y} - q &\leq \norm{T^{m_n}x - T^{m_n-1}x} + \norm{T^{m_n-1}x-y}-q\\
	&= \norm{T^{m_n}x - T^{m_n-1}x} + \norm{T^{k_{j_n}}x - y} - q\\
	&< \frac{\e}{2} + \frac{\e}{2} = \e
\end{align*}

Entirely similarly, we have that
\begin{align*}
\norm{T^{m_n}x - y} - q &\geq -\norm{T^{m_n}x - T^{m_n-1}x} + \norm{T^{m_n-1}x-y}-q\\
	&= -\norm{T^{m_n}x - T^{m_n-1}x} + \norm{T^{k_{j_n}}x - y} - q\\
	&> -\frac{\e}{2} - \frac{\e}{2} = -\e
\end{align*}

Hence, $\big| \, \norm{T^{m_n}x -y } - q \, \big| < \e$ for $n$ large enough.  Since $\{m_n : n \in \N \} \cup \{k_n : n \in \N\} = \N$, we have finally that $\lim_n \norm{T^nx - y}$ exists for any $y \in \omega_w (x)$.
\end{proof}

\begin{remark}\label{remark}
The above argument presented in the proof above actually works for any $y \in F(T)$, but in particular for $y \in \omega_w (x)$.  This will be of use to us in Theorem \ref{not bounded}.
\end{remark}

\begin{proof}[Proof of Lemma \ref{uniqueness}]
Suppose for a contradiction that $z$ and $y$ are distinct elements of $\omega_w (x)$.  Then there exist $(n_k)_k$ and $(m_k)_k$ for which $T^{n_k}x \weak_k z$ and $T^{m_k}x \weak_k y$.  Thus, using the fact that $C$ is Opial, we have
\begin{align*}
\lim_n \norm{T^nx - y} &= \lim_n \norm{T^{m_k}x-y}\\
	&< \lim_n \norm{T^{m_k}x - z}\\
	&= \lim_n \norm{T^n x - z}\\
	&= \lim_n \norm{T^{n_k}x-z}\\
	&< \lim_n \norm{T^{n_k}x-y}\\
	&= \lim_n \norm{T^n x - y},
\end{align*}
which is a contradiction.  Thus, $\omega_w (x)$ is a singleton.
\end{proof}

\begin{proof}[Proof of Theorem \ref{theorem}]

As stated above,  let $\omega_w (x) := \{y \in C : y \mbox{ is a weak subsequential limit of } (T^nx)_n \}$ and note that $\omega_w(x) \neq \emptyset$ since $C$ is weakly compact, as well as that the demiclosedness of $I-T$ at $0$ gives us that $\omega_w (x) \subseteq F(T)$.  By Lemma \ref{uniqueness}, we know that $\omega_w (x)$ consists of a single point, say $y$.  Thus, $T^nx \weak_n y$, and the theorem is proved.
\end{proof}


\section{Results for arbitrary $\al$}
We have the corresponding theorem for $\al$ of arbitrary length.

\begin{theorem}
If $C \subseteq X$ is weakly compact, convex, and has the Opial property, $T: C \to C$ is $\al$-nonexpansive and asymptotically regular at some point $x \in C$, then $T^nx$ converges weakly to a fixed point of $T$.
\end{theorem}

The theorem will follow immediately from the analogous lemma concerning convergence of the sequence $(\norm{T^nx-y})_n$ for any $y \in \omega_w (x)$.

\begin{lemma}
Suppose $C$ is weakly compact and convex with the Opial property, and suppose that $T : C \to C$ is $\al$-nonexpansive and asymptotically regular at some $x \in C$.  Then for all $y \in \omega_w (x)$, $\lim_n \norm{T^nx - y}$ exists.
\end{lemma}

\begin{proof}[Proof of the Lemma]
Let $\al = (\al_1,\ldots, \al_{n_0})$.  In the same way as above, we build a sequence $(k_n)_n$ for which
\begin{enumerate}
\item $k_n + 1 \leq k_{n+1} \leq k_n + n_0$, and
\item $\norm{T^{k_{n+1}}x-y} \leq \norm{T^{k_n}x - y}$
\end{enumerate}

Again, as above, let $M = \N \setminus \{ k_n : n \in \N \}$.  If $M$ is finite, we are done.  If $M$ is infinite, then write the elements of $M$ as $(m_n)_n$, strictly increasing.  Note that for all $n \in \N$, there exist $j_n \in \N$ and $i_n \in \{1, \ldots, n_0-1\}$ for which
\[
m_n = k_{j_n} + i_n
\]
Also, $(j_n)_n$ is strictly increasing.  Now, for any $\e>0$, we can find $n$ large enough so that
\[
\norm{T^{m_n - j + 1} x - T^{m_n - j } x} < \frac{\e}{n_0} \quad \mbox{for all } j = 1, \ldots, n_0-1, \mbox{ and}
\]
\[
\Big|  \norm{T^{k_{j_n}}x - y} - q \, \Big|< \frac{\e}{n_0}, \quad \mbox{where } q = \lim_{n\to\ii} \norm{T^{k_n}x-y}
\]

Thus, for $n$ large, we have
\begin{align*}
\norm{T^{m_n}x - y} - q &\leq \norm{T^{m_n}x - T^{m_n - 1} x} + \cdots + \norm{T^{m_n - i_n + 1}x - T^{m_n - i_n} x} + \norm{T^{m_n - i_n} x - y} - q\\
	&= \norm{T^{m_n}x - T^{m_n - 1} x} + \cdots + \norm{T^{m_n - i_n + 1}x - T^{m_n - i_n} x} + \norm{T^{k_{j_n}} x - y} - q\\
	&< \underbrace{\frac{\e}{n_0} + \cdots + \frac{\e}{n_0}}_{i_n \mbox{ times}} + \frac{\e}{n_0}\\
	&\leq (n_0-1) \frac{\e}{n_0} + \frac{\e}{n_0} = \e
\end{align*}
A similar argument proves that $\big| \norm{T^{m_n}x-y} - q \, \big| < \e$ for $n$ large, and the lemma is proved.
\end{proof}


\section{Losing boundedness of $C$}
Similar arguments show that, under appropriate circumstances, the assumption of boundedness of $C$ may be dropped.  Before we state the theorem, we need the notion of a duality mapping, a lemma due to Opial \cite[Lemma 3]{opial67}, and a theorem of Garc\'ia and Piasecki \cite[Theorem 4.2]{garciapiasecki12}.

\begin{definition}
A mapping $J: X \to X^*$ is called a \textit{duality mapping} of $X$ into $X^*$ with gauge function $\mu$ (that is, $\mu: [0,\ii) \to [0,\ii)$ is strictly increasing, continuous, and $\mu(0)=0$) if, for every $x \in X$, $(Jx)(x) = \norm{Jx}\norm{x} = \mu(\norm{x})\norm{x}$.
\end{definition}

Recall also that a Banach space $(X, \norm{\cdot})$ is called \textit{uniformly convex} if for every $\e \in (0,2]$, there exists a $\delta \in (0,1)$ such that
\[
\begin{cases}
\norm{u},\norm{v} \leq 1,&\\
\norm{u-v} \geq \e
\end{cases}
\quad \implies \quad \frac{1}{2} \norm{u +v} \leq 1-\delta.
\]
It is easy to see that this is equivalent to a sequential notion of uniform convexity.  That is, $X$ is uniformly convex if and only if for every $R>0$ and for any sequences $(u_n)_n$ and $(v_n)_n$ in $X$,
\[
\begin{cases}
\norm{u_n},\norm{v_n} \leq R \mbox{ for all } $n$, \mbox{ and }&\\
\frac{1}{2}\norm{u_n+v_n} \to R
\end{cases}
\quad \implies \quad \norm{u_n - v_n} \to 0.
\]

Now we have a lemma of Opial describing those uniformly convex spaces which have Opial's property. 

\begin{lemma}[Opial]
If $(X, \norm{\cdot})$ is uniformly convex and has a weakly continuous duality mapping, then $(X,\norm{\cdot})$ is Opial.
\end{lemma}

Finally, we state a theorem of Garc\'ia and Piasecki regarding the structure of the set of fixed points for any mean nonexpansive mapping defined in a strictly convex space.

\begin{theorem}[Garc\'ia and Piasecki]
Suppose $C \subseteq X$ is closed and convex and $(X,\norm{\cdot})$ is strictly convex.  Then for any mean nonexpansive mapping $T : C\to C$, $F(T)$ is closed and convex.
\end{theorem}

We use the tools above to prove a theorem:

\begin{theorem}\label{not bounded}
Suppose $(X,\norm{\cdot})$ is uniformly convex with a weakly sequentially continuous duality map and $C \subseteq X$ is closed and convex.  Assume further that $T : C\to C$ is $\al$-nonexpansive, $F(T) \neq \emptyset$, and $T$ is asymptotically regular at some $x \in C$.  Then $(T^n x)_n$ converges weakly to some $y_0 \in F(T)$.
\end{theorem}

The proof follows largely from the work done above and the original proof for nonexpansive mappings due to Opial \cite[Theorem 1]{opial67}, and we present it here for completeness.

\begin{proof}
By Opial's lemma, $X$ is uniformly convex with a weakly continuous duality map implies that $X$ is Opial.  Thus, for every $y \in F(T)$, by the proof of Theorem \ref{existence} and Remark \ref{remark}, we know that $\lim_n \norm{T^nx - y}$ exists.  In particular, this implies that $\{T^nx : n \in \N \}$ is bounded.  Let $\p: F(T) \to [0,\ii)$ be given by $\p(y) := \lim_n \norm{T^n x - y}$.  For any $r \in [0,\ii)$, consider the set
\begin{align*}
F_r :&= \{ y \in F(T) : \p(y) \leq r\}\\
	&= \p^{-1} [0, r].
\end{align*}
We summarize the relevant facts about $F_r$.
\begin{claim}\label{claim}
The sets $F_r$ satisfy the following four properties:
\begin{enumerate}
\item $F_r$ is nonempty for $r$ sufficiently large, 
\item $F_r$ is closed, bounded, and convex for all $r\geq 0$,
\item there is a minimal $r_0$ for which $F_{r_0}$ is nonempty, and
\item $F_{r_0}$ is a singleton.
\end{enumerate}
\end{claim}
\begin{proof}[Proof of Claim \ref{claim}]
(1) and (2) are easy to verify.  

(3) follows from the fact that each $F_r$ is weakly compact (since $X$ is reflexive) and $\{F_r : r \geq 0\}$ forms a nested family.  Thus, if each $F_r \neq \emptyset$ for $r > t$ for some $t \geq 0$, it follows that 
\[
F_t = \bigcap_{r>t} F_r \neq \emptyset.
\]  

(4) follows from uniform convexity.  Suppose $u, v \in F_{r_0}$ with $u \neq v$, and let $z := \frac{1}{2}(u+v)$.  Note that $z \in F_{r_0}$ since $F_{r_0}$ is convex.  Because $r_0$ is minimal for which $F_{r_0} \neq \emptyset$, it follows that $\p(u)=r_0=\p(v)$.  We want to show that $\p(z) < r_0$.  Suppose for a contradiction that $\p(z) = r_0$.  Then
\[
\lim_n \frac{1}{2} \norm{(T^nx - u) + (T^n x - v)} = \lim_n \norm{T^nx - z} = r_0
\]
and uniform convexity implies that
\[
\lim_n \norm{(T^nx - u) - (T^n x - v)} = \norm{u - v} = 0,
\]
but $\norm{u-v}>0$.  This tells us that $\p(z) < r_0$, which contradicts the minimality of $r_0$. Hence, $F_{r_0}$ must be a singleton.  This completes the proof of the claim.
\end{proof}

Let $F_{r_0} = \{y_0\}$.  We aim to show that $T^nx \weak y_0$.  For a contradiction, suppose this is not the case.  Since $\{ T^n x : n \in \N \}$ is bounded and $X$ is reflexive, there is some subsequence $(T^{n_k}x)_k$ converging weakly to some $y \neq y_0$.  By asymptotic regularity of $T$ and demiclosedness of $I-T$ at $0$, we know that $\norm{(I-T)T^{n_k}x} \to 0$ yields $Ty=y$.  That is, $y \in F(T)$.  Thus,
\begin{align*}
r_0 = \p(y_0) &= \lim_n \norm{T^nx - y_0}\\
	&= \lim_k \norm{T^{n_k}x - y_0}\\
	&> \lim_k \norm{T^{n_k}x - y}\\
	&= \lim_n \norm{T^nx - y} = \p(y),
\end{align*}
which contradicts the minimality of $r_0$.  Finally, we have that $T^nx \weak y_0$, and the proof is complete.
\end{proof}

\begin{remark}
We note here, just as Opial did, that the same result will hold in any reflexive Opial space where $F(T)$ is convex and $F_{r_0}$ is a singleton.  For example, to guarantee that $F(T)$ is convex for a mean nonexpansive map, we need only assume strict convexity of $X$ as opposed to uniform convexity.
\end{remark}


\begin{thebibliography}{20}
\bibitem{browderpetryshyn66}
	F.E. Browder, W.V. Petryshyn,
	``The solution by iteration of nonlinear functional equations in Banach spaces,'' 
	Bull. Amer. Math. Soc., \textbf{72} (1966), pp. 571-575.
	
\bibitem{gallagher16}
	T.M. Gallagher,
	``The demiclosedness principle for mean nonexpansive mappings,''
	Journal of Math. Analysis \& Appl., \textbf{439} (2016), pp. 832-842.

\bibitem{garciapiasecki12}
	V.P. Garc\'ia and \L. Piasecki,
	``On mean nonexpansive mappings and the Lifshitz constant,''
	Journal of Math. Analysis \& Appl., \textbf{396} (2012), pp. 448-454.

\bibitem{gjp07}
	K. Goebel and M. Jap\'on Pineda,
	``A new type of nonexpansiveness,''
	Proc. of the 8th International Conference on Fixed Point Theory and Appl.,
	Chiang Mai, 2007

\bibitem{goebelsims10}
	K. Goebel and B. Sims,
	\textit{Mean Lipschitzian Mappings},
	Contemp. Math.
	\textbf{513}, 157-167 (2010)

\bibitem{opial67}
	Z.O. Opial,
	``Weak convergence of the sequence of successive approximations for nonexpansive mappings,''
	Bull. AMS, \textbf{73} (1967), pp. 591-597

\bibitem{piasecki13}
	 {\L}. Piasecki,
	\textit{Classification of Lipschitz Mappings}, 
	CRC Press, 
	2013


\end{thebibliography}
\end{document}